\documentclass[12pt]{article}
\usepackage{a4wide}
\usepackage{graphicx}
\usepackage{amsmath}
\usepackage{amssymb}
\usepackage{amsthm}
\usepackage{enumerate}
\usepackage{combelow}

\newtheorem{theorem}{Theorem}[section]
\newtheorem{definition}[theorem]{Definition}

\newtheorem{lemma}[theorem]{Lemma}

\newtheorem{corollary}[theorem]{Corollary}
\newtheorem{conjecture}[theorem]{Conjecture}

\newtheorem{problem}[theorem]{Problem}

\newcommand{\comment}[1]{}

\begin{document}

\title{Highly linked tournaments}

\author{\large{Alexey Pokrovskiy} 
\\
\\ Methods for Discrete Structures,
\\ {Freie Universit\"at,} 
\\ Berlin, Germany.
\\ {Email: \texttt{alja123@gmail.com}}
}

\maketitle

\begin{abstract}
A (possibly directed) graph is $k$-linked if for any two disjoint sets of vertices $\{x_1, \dots, x_k\}$ and $\{y_1, \dots, y_k\}$ there are vertex disjoint paths $P_1, \dots, P_k$ such that $P_i$ goes from $x_i$ to~$y_{i}$.  A theorem of Bollob\'as and Thomason
says that every $22k$-connected (undirected) graph is $k$-linked. It is
desirable to obtain analogues for directed graphs as well. Although
Thomassen showed that the Bollob\'as-Thomason Theorem does not hold for
general directed graphs, he proved an analogue of the theorem for tournaments---there is a function $f(k)$ such that every strongly $f(k)$-connected tournament is $k$-linked.
The bound on $f(k)$ was reduced to $O(k \log k)$ by  K\"uhn, Lapinskas, Osthus,
and Patel, who also conjectured that a linear bound should hold. We prove this conjecture, by showing that every strongly $452k$-connected tournament is $k$-linked.
\end{abstract}

\section{Introduction}
A graph is connected if there is a path between any two vertices.
A graph is $k$-connected if it remains connected after the removal of any set of  $(k-1)$-vertices. This could be seen as a notion of how robust the graph is. For example, if the graph represents a communication network, then the connectedness measures how many nodes need to fail before communication becomes impossible.

Similar notions make sense for directed graphs, except in that context we usually want a \emph{directed path} between every pair of vertices. If this holds, we say that the directed graph is \emph{strongly connected}.  A directed graph is strongly $k$-connected it remains strongly connected after the removal of any set of  $(k-1)$-vertices. In this paper, when dealing with connectedness of directed graphs we will always mean strong connectedness.

Connectedness is a fundamental notion in graph theory, and there are countless theorems which involve it. Perhaps the most important of these is Menger's Theorem, which provides an alternative characterization of $k$-connectedness. Menger's Theorem says that a graph is $k$-connected if, and only if, there are $k$ internally vertex-disjoint paths between any pair of vertices.
Menger's Theorem has the following simple corollary: 
\begin{corollary}\label{MengerCorollary}
If $G$ is $k$-connected then for any two disjoint sets of vertices $\{x_1, \dots, x_k\}$ and $\{y_1, \dots, y_k\}$ there are vertex-disjoint paths $P_1, \dots, P_k$ such that $P_i$ goes from $x_i$ to $y_{\sigma(i)}$ for some permutation $\sigma$ of $[k]$.
\end{corollary}
This corollary is proved by constructing a new graph $H$ from $G$ by adding two vertices $x$ and $y$ such that $x$ is joined to $\{x_1, \dots, x_k\}$ and $y$ is joined to $\{y_1, \dots, y_k\}$. It is easy to see that $H$ is $k$-connected, and so, by Menger's Theorem, has $k$ vertex-disjoint $x$ -- $y$ paths. Removing the vertices $x$ and $y$ produces the required paths $P_1, \dots, P_k$. It is not hard to see that the converse of Corollary~\ref{MengerCorollary} holds for  graphs on at least $2k$ vertices as well.

Notice that in Corollary~\ref{MengerCorollary}, we had no control over where the path $P_i$ starting at $x_i$ ends up---it could end at any of the vertices $y_1, \dots, y_k$. In practice we might want to have control over this. This leads to the notion of $k$-linkedness. A graph is $k$-linked if for any two disjoint sets of vertices $\{x_1, \dots, x_k\}$ and $\{y_1, \dots, y_k\}$ there are vertex disjoint paths $P_1, \dots, P_k$ such that $P_i$ goes from $x_i$ to $y_{i}$. 

Linkedness is a stronger notion than connectedness. 
A natural question is whether a $k$-connected graph must also be $\ell$-linked for some $\ell$ (which may be smaller than $k$). Larman and Mani \cite{LM}, and  Jung \cite{Jung} were the first to show that this is indeed the case---they showed that there is a function $f(k)$ such that every $f(k)$-connected graph is $k$-linked. This result uses a theorem of Mader~\cite{Mader} about the existence of large topological complete minors in graphs with many edges. The first bounds on $f(k)$ were exponential in $k$, but
Bollob\'as and Thomason showed that a linear bound on the connectedness suffices \cite{BT}. 
\begin{theorem}[Bollob\'as and Thomason]
\label{BollobasThomasson}
Every $22 k$-connected graph is $k$-linked.
\end{theorem}
The constant $22$ has since been reduced to $10$ by Thomas and Wollan \cite{TW}.

Much of the above discussion holds true for directed graphs as well (when talking about \emph{strong} $k$-connectedness and \emph{directed} paths). Menger's Theorem remains true, as does Corollary~\ref{MengerCorollary}. A directed graph is $k$-linked if for two disjoint sets of vertices $\{x_1, \dots, x_k\}$ and $\{y_1, \dots, y_k\}$ there are vertex disjoint directed paths $P_1, \dots, P_k$ such that $P_i$ goes from $x_i$ to~$y_{i}$. Somewhat surprisingly there is no function $f(k)$ such that every strongly $f(k)$-connected directed graph is $k$-linked. Indeed Thomassen constructed directed graphs of arbitrarily high connectedness which are not even $2$-linked \cite{ThomassenLinkedDigraphs}. Thus, there is a real difference between the directed and undirected cases. 
For tournaments however the situation is better (A tournament is a directed graph which has exactly one directed edge between any two vertices). There Thomassen showed that there is a constant $C$, such that every $Ck!$-connected tournament is $k$-linked \cite{ThomassenLinkedTournaments}. K\"uhn, Lapinskas, Osthus, and Patel improved the bound on the connectivity to $10^4k\log k$.
\begin{theorem}[K\"uhn, Lapinskas, Osthus, and Patel, \cite{KLOP}]\label{LinkingTheoremKLOP}
All strongly $10^4k\log k$-connected tournaments are $k$-linked.
\end{theorem}
This theorem is proved using a beautiful construction utilizing the asymptotically optimal sorting networks of Ajtai, Koml\'os, and Szemer\'edi \cite{AKS}. The proof is based on building a small sorting network inside the tournament, which is combined with the directed version of Corollary~\ref{MengerCorollary} in order to reorder the endpoints of the paths so that $P_i$ goes from $x_i$ to $y_i$. We refer to \cite{KLOP} for details.

Since sorting networks on $k$ inputs require size at least $k\log k$, it is unlikely that this approach can give a $o(k\log k)$ bound in Theorem~\ref{LinkingTheoremKLOP}.
Nevertheless, K\"uhn, Lapinskas, Osthus, and Patel conjectured that a linear bound should be possible.
\begin{conjecture}[K\"uhn, Lapinskas, Osthus, and Patel, \cite{KLOP}]\label{LinkingConjecture}
There is a constant $C$ such that every strongly $Ck$-connected tournament is $k$-linked.
\end{conjecture}
There has also been some work for small $k$. Bang-Jensen showed that every $5$-connected tournament is $2$-linked \cite{BangJensen}. Here the value ``$5$'' is optimal.

The main result of this paper is a proof of Conjecture~\ref{LinkingConjecture}.
\begin{theorem}\label{LinkingTheorem}
Every strongly $452k$-connected tournament is $k$-linked.
\end{theorem}

The above theorem is proved using the method of ``linkage structures in tournaments'' recently introduced in~\cite{KLOP} and~\cite{KOT}. Informally, a linkage structure $L$ in a tournament $T$, is a small subset of $V(T)$ with the property that for many pairs of vertices $x,y$ outside $L$, there is a path from $x$ to $y$ most of whose vertices are contained in $L$. Such structures can be found in highly connected tournaments, and they have various applications such as finding Hamiltonian cycles~\cite{KLOP, PokrovskiyHamiltonian} or partitioning tournaments into highly connected subgraphs~\cite{KOT}.
Linkage structures were introduced in the same paper where Conjecture~\ref{LinkingConjecture} was made. However in the past they were constructed \emph{using Theorem~\ref{LinkingTheoremKLOP}} to first show that a tournament is highly linked. In our paper the perspective is different---the linkage structures are built using only connectedness, and then linkedness follows as a corollary of the presence of the linkage structures.

It would be interesting to reduce the constant $452$ in Theorem~\ref{LinkingTheorem}. 
It is not hard to find minor improvements to our proof in Section 2 which improve this constant by a little bit. It is not clear what the correct value of the constant should be, and we are not aware of any non-trivial constructions for large $k$.
In view of the Bollob\'as-Thomason Theorem, we pose the following problem.
\begin{problem}
Show that every strongly $22k$-connected tournament is $k$-linked.
\end{problem}

\section{Proof of Theorem~\ref{LinkingTheorem}}
A directed path $P$ is a sequence of vertices $v_1, v_2,\dots, v_k$ in a directed graph such that $v_i v_{i+1}$ is an edge for all $i=1, \dots, k-1$. The vertex $v_1$ is called the \emph{start} of $P$, and $v_k$ the \emph{end} of $P$. The \emph{length} of $P$ is the number of edges it has which is $|P|-1$. The vertices $v_{2}, \dots, v_{k-1}$ are the \emph{internal vertices} of $P$. Two paths are said to be internally disjoint if their internal vertices are distinct.

A tournament $T$ is transitive if for any three vertices $x,y,z\in V(T)$, if $xy$ and $yz$ are both edges, then $xz$ is also an edge. It's easy to see that a tournament is transitive exactly when it has an ordering $(v_1, v_2,\dots, v_k)$ of $V(T)$ such that the edges of $T$ are $\{v_iv_j: i<j\}$. We say that $v_1$ is the \emph{tail} of $T$, and $v_k$ is the \emph{head} of $T$.

The \emph{out-neighbourhood} of a vertex $v$ in a directed graph, denoted $N^+(v)$, is the set of vertices $u$ for which $vu$ is an edge. Similarly, the \emph{in-neighbourhood}, denoted $N^-(v)$, is the set of vertices $u$ for which $uv$ is an edge.
The \emph{out-degree} of $v$ is $d^+(v)=|N^+(v)|$, and the \emph{in-degree} of $v$ is $d^-(v)=|N^-(v)|$. A useful fact is that every tournament, $T$, has a vertex of out-degree at least $(|T|-1)/2$, and a vertex of in-degree at least $(|T|-1)/2$. To see this, observe that since $T$ has $\binom{|T|}2$ edges, its average in and out-degrees are both $(|T|-1)/2$.

We'll need the following lemma which says that in any tournament, we can find two large sets such that there is a linkage between them.
\begin{lemma}\label{LargeLinkage}
Let  $n$ and $m$ be two integers with $m\leq n/11$.  
Every tournament $T$ on $n$ vertices, contains two disjoint sets of vertices $\{x_1, \dots, x_{m}\}$ and $\{y_1, \dots, y_{m}\}$ such that for any permutation $\sigma$ of $[m]$, there are vertex-disjoint paths $P_1, \dots, P_{m}$ such that $P_i$ goes from $x_i$ to~$y_{\sigma(i)}$.
\end{lemma}
\begin{proof}
Let $x_1, \dots, x_{m}$ be a set of $m$ vertices in $T$ of largest out-degrees i.e. any set such that any vertex $u$ outside it satisfies $d^+(u)\leq d^+(x_i)$ for all $i$.
Let $y_1, \dots, y_{m}$ be a set of $m$ vertices in $T$ of largest in-degrees. Since $m\leq n/11$, we can choose $\{x_1, \dots, x_{m}\}$ and $\{y_1, \dots, y_{m}\}$ to be disjoint.

Recall that every tournament $T$ has a vertex of out-degree at least $(|T|-1)/2$. This means that $d^+(x_i)\geq ({n}-m)/2$ for each $i=1, \dots, m$ (since otherwise, there would be a vertex in $T\setminus\{x_1, \dots, x_m\}+x_i$ of out-degree larger than $x_i$, contradicting the choice of $x_i$). Similarly, we obtain  $d^-(y_i)\geq ({n}-m)/2$ for each $i=1, \dots, m$.

For each $i$ and $j\leq m$, let $X_{i,j}=\big(N^+(x_i)+x_i\big)\setminus \big(N^-(y_j)+y_j\big)$, $Y_{i,j}=\big(N^-(y_j)+y_j\big)\setminus \big(N^+(x_i)+x_i\big)$, $I_{i,j}=N^+(x_i)\cap N^-(y_j)$, and $M_{i,j}$ a maximum matching of edges directed from $X_{i,j}$ to $Y_{i,j}$.

Notice that we have 
$$|X_{i,j}\setminus V(M_{i,j})|=|N^+(x_i)+x_i-y_i|-|I_{i,j}|-e(M_{i,j})\geq \frac{1}{2}({n}-m)-|I_{i,j}|-e(M_{i,j}).$$
Similarly, we obtain $|Y_{i,j}\setminus V(M_{i,j})| \geq \frac{1}{2}({n}-m)-|I_{i,j}|-e(M_{i,j})$.

Since $M$ is maximal, all the edges between $X_{i,j}\setminus V(M_{i,j})$ and $Y_{i,j}\setminus V(M_{i,j})$ go from $Y_{i,j}$ to~$X_{i,j}$. Therefore, if  $\frac{1}{2}({n}-m)-|I_{i,j}|-e(M_{i,j})\geq m$ holds, then the lemma follows by choosing $x'_1, \dots, x'_m$ to be any $m$ vertices in $Y_{i,j}\setminus V(M_{i,j})$, and $y'_1, \dots, y'_m$ to be any $m$ vertices in $X_{i,j}\setminus V(M_{i,j})$. This  ensures that we can always choose length $1$ paths $P_1, \dots, P_m$ as in the lemma.

Therefore, we can suppose that $\frac{1}{2}({n}-m)-|I_{i,j}|-e(M_{i,j})< m$. Combining this with $m\leq {n}/11$ we obtain that $|I_{i,j}|+e(M_{i,j})>4m$ for every $i$ and $j$.

Notice that for all $i$ and $j$, there are $|I_{i,j}|+e(M_{i,j})\geq 4m+1$ internally vertex disjoint paths of length $\leq 3$ between $x_i$ and $y_j$. This allows us to construct vertex disjoint paths $P_1, \dots, P_m$ each of length $\leq 3$, such that $P_i$ goes from $x_i$ to $y_{\sigma(i)}$ (where $\sigma$ is an arbitrary permutation of $[m]$). Indeed assuming we have constructed the paths $P_1, \dots, P_{k}$, then we have $|V(P_1)\cup\dots\cup V(P_{k})|\leq 4m$, and so one of the $4m+1$ internally vertex disjoint paths between $x_{k+1}$ and $y_{k+1}$ must be disjoint from $V(P_1)\cup\dots\cup V(P_{k})$. We let $P_{k+1}$ be this path, and then repeat this process until we have the required $m$ paths.
\end{proof}

A set of vertices $S$ in-dominates another set $B$, if for every $b\in B\setminus S$, there is some $s\in S$ such that $bs$ is an edge. Notice that by this definition, a set in-dominates itself. A \emph{in-dominating set} in a tournament $T$ is any set $S$ which in-dominates $V(T)$. Notice that by repeatedly pulling out vertices of largest in-degree and their in-neighbourhoods from $T$, we can find an in-dominating set of order at most $\lceil\log_2 |T|\rceil$. For our purposes we'll study sets which are constructed by pulling out some fixed number of vertices by this process.

\begin{definition}
We say that a sequence $(v_1,v_2,\dots, v_k)$ of vertices of a tournament $T$ is a partial greedy in-dominating set if $v_1$ is a maximum in-degree vertex in $T$, and for each $i$, $v_i$ is a maximum in-degree vertex in the subtournament of $T$ on $N^+(v_1)\cap N^+(v_2)\cap \dots \cup N^+(v_{i-1})$.
\end{definition}
Partial greedy out-dominating sets are defined similarly, by letting $v_i$ be a maximum out-degree vertex in $N^-(v_1)\cap N^-(v_2)\cap \dots \cap N^-(v_{i-1})$ at each step.

Notice that every partial greedy in-dominating set is a transitive tournament with head $v_k$ and tail $v_1$.

For small $k$, partial greedy in-dominating sets do not necessarily dominate all the vertices in a tournament. A crucial property of partial greedy in-dominating sets is that the vertices they don't dominate have large out-degree.
The following is a version of a lemma appearing in~\cite{KLOP}.
\begin{lemma}\label{GreedyDominatingSet}
Let $(v_1,v_2,\dots, v_k)$ be a partial greedy in-dominating set in a tournament~$T$. Let $E$ be the set of vertices which are not in-dominated by $A$. Then every $u\in E$ satisfies $d^+(u)\geq 2^{k-1}|E|.$
\end{lemma}
\begin{proof}
The proof is by induction on $k$. The initial case is when $k=1$. In this case we have $E=N^+(v_1)$ where $v_1$ is a maximum in-degree vertex in $T$. For any $u\in E$, we must have $d^-(u)\leq d^-(v_1)=|T\setminus E- v_1|=|T|-|E|-1$. Therefore we have $d^+(u)=|T\setminus N^-(u)-u|= |T|-d^-(u)-1\geq |E|$ as required.

Now suppose that the lemma holds for $k=k_0$. Let $(v_1, \dots, v_{k_0+1})$ be a partial greedy in-dominating set in $T$, and let $E_0=N^+(v_1)\cap \dots \cap N^+(v_{k_0})$. 
By induction we have $d^+(u)\geq 2^{k_0-1}|E_0|$ for every $u\in E_0$. 
By definition $v_{k_0+1}$ is a maximum in-degree vertex in $E_0$. Let $E=E_0\cap N^+(v_{k_0+1})$ be the set of vertices not in-dominated by $(v_1, \dots, v_{k_0+1})$. 
Since $v_{k_0+1}$ is a maximum in-degree vertex in $E_0$, we have $|N^-(v_{k_0+1})\cap E_0|\geq (|E_0|-1)/2$ which implies $|E|=|E_0|-|(N^-(v_{k_0+1})+v_{k_0+1})\cap E_0|\leq |E_0|/2$.
Combining this with the inductive hypothesis, we obtain $d^+(u)\geq 2^{k_0-1}|E_0|\geq 2^{k_0}|E|$, completing the proof. 
\end{proof}

We are now ready to prove the main result of this paper.
\begin{proof}[Proof of Theorem~\ref{LinkingTheorem}.]
Let $T$ be a strongly $452k$-connected tournament. Notice that this means that all vertices in $T$ have in-degree and out-degree at least $452k$.

Let $x_1, \dots, x_k$ and $y_1, \dots, y_k$ be vertices in $T$ as in the definition of $k$-linkedness.  We will construct vertex disjoint paths from $x_i$ to $y_i$. Let $T'=T\setminus \{x_1, \dots, x_k, y_1, \dots, y_k\}$.

Let $D^-_1$ be a partial greedy in-dominating set in $T'$ on $2$ vertices. Then, for all $i=2, \dots,55k$, let $D^-_i$ be a partial greedy in-dominating set on $2$ vertices in $T'\setminus (D^-_1\cup \dots\cup D^-_{i-1})$.

Similarly, let $D^+_1$ be a partial greedy out-dominating set on $2$ vertices in $T'\setminus (D^-_1\cup \dots\cup D^-_{55k})$. Then, for all $i=2, \dots, 55k$, let $D^+_i$ be a partial greedy out-dominating set  on $2$ vertices in $T'\setminus (D^+_1\cup \dots\cup D^+_{i-1}\cup D^-_1\cup \dots\cup D^-_{55k})$.

Let $X= D^+_1\cup \dots\cup D^+_{55k}\cup D^-_1\cup \dots\cup D^-_{55k}\cup\{x_1,\dots, x_k, y_1, \dots, y_k\}$. 
For each $i$, let $E^-_i$ be the set of vertices in $T\setminus X$ which aren't in-dominated by $D^-_i$, and $E^+_i$ the set of vertices in $T\setminus X$ which aren't out-dominated by $D^+_i$. By  Lemma~\ref{GreedyDominatingSet}, we have $d^+(v)\geq 2|E^-_i|$ for every $v\in E^-_i$, and also  $d^-(v)\geq 2|E^+_i|$ for every $v\in E^+_i$. 

Let $T^-$ be the set of heads of $D^-_1,\dots, D^-_{55k}$, and $T^+$ the set of tails of $D^+_1,\dots, D^+_{55k}$.
 Apply Lemma~\ref{LargeLinkage} to $T^-$ in order to find two subsets $X^-$ and $Y^-$ of order $5k$ of $V(T^-)$, such that for any bijection $f:X^-\to Y^-$, there is a set of $5k$ vertex-disjoint paths in $T^-$ with each path joining $x$ to $f(x)$ for some $x\in X^-$.
 Apply Lemma~\ref{LargeLinkage} to $T^+$ in order to find two subsets $X^+$ and $Y^+$ of order $5k$ of $V(T^+)$, such that for any bijection $f:X^+\to Y^+$, there is a set of $5k$ vertex-disjoint paths in $T^+$ with each path joining $x$ to $f(x)$ for some $x\in X^+$.
 Reorder $(D^-_1, \dots, D^-_{55k})$ so that $X^-$ is the set of heads of $D^-_1, \dots, D^-_{5k}$. 
 Reorder $(D^+_1, \dots, D^+_{55k})$ so that $Y^+$ is the set of tails of $D^+_1, \dots, D^+_{5k}$.
 Notice that since each partial greedy dominating set is on $2$ vertices, we have $|X|\leq 222k$.
By Menger's Theorem, since $T$ is $452$-connected, there is a set of vertex-disjoint paths $Q_1, \dots, Q_{5k}$ in $(T\setminus X)\cup Y^-\cup X^+$ such that each path $Q_i$ starts in $Y^-$ and ends in $X^+$.

Recall that all out-degrees in $T$ are at least $452k$ and $|X|\leq 222k$.
Therefore, for each $i=1, \dots, k$ we can choose an out-neighbour $x'_i$ of $x_i$ which is not in $X$. Similarly for each $i$ we can choose an in-neighbour $y'_i$ of $y_i$ which is not in $X$. In addition we can ensure that $x'_1, \dots, x'_k,y'_1, \dots, y'_k$ are all distinct. Let $X'=X\cup\{x'_1, \dots, x'_k,y'_1, \dots, y'_k\}$.

Notice that each vertex $v\in E^-_i$ satisfies $d^+(v)\geq 2|E^-_i|$ and $2|X'|+4k$. Averaging these, we get $d^+(v)\geq |E^-_i| + |X'|+2k$  and so $v$ has  at least $2k$ out-neighbours outside of $E^-_i \cup X'$. Similarly each $v\in E^+_i$ has at least  $2k$ in-neighbours outside of $E^+_i \cup X'$.
Therefore, for each $i$, we choose $x''_i$ to be either equal to $x'_i$ if $x'_i\not\in E^-_i$ or we choose $x''_i$ to be an out-neighbour of $x'_i$ in $T\setminus (E^-_i\cup X')$. Similarly, for each $i$, we choose $y''_i$ to be either equal to $y'_i$ if $y'_i\not\in E^+_i$ or we choose $y''_i$ to be an in-neighbour of $y'_i$ in $T\setminus (E^+_i\cup X')$. We can also choose  the vertices  $x''_1, \dots, x''_k,y''_1, \dots, y''_k$ so that they are all distinct (since when $x''\neq x'$ and $y''\neq y'$ are always at least $2k$ choices for $x''_i$ and $y''_i$ respectively).

For each $i=1, \dots, k$, let $Q^-_i$ be a path from $x''_i$ to the head of $D^-_i$ whose internal vertices are all in $D^-_i$. The facts that $D^-_i$ is transitive and $x''_i\not\in E^-_i$ ensure that we can do this.  Similarly, for each $i$ let $Q^+_i$ be a path from  the tail of $D^+_i$ to $y''_i$ whose internal vertices are all in $D^+_i$.

Notice that at least $k$ of the paths $Q_1, \dots, Q_{5k}$ are disjoint from $\{x'_1, \dots, x'_k,$ $y'_1, \dots, y'_k,$ $x''_1, \dots, x''_k,$ $y''_1, \dots, y''_k\}$. Let $Q'_1, \dots, Q'_k$ be some choice of such paths.

Since $Q^-_i$ ends in $X^-$ and $Q'_i$ starts in $Y^-$, Lemma~\ref{LargeLinkage} implies that we can choose disjoint paths $P^-_1,\dots, P^-_k$ in $T^-$ such that $P^-_i$ is from the end of $Q^-_i$ to the start of $Q'_i$. 
Similarly we can choose disjoint paths $P^+_1,\dots, P^+_k$ in $T^+$ such that $P^+_i$ is from the end of $Q'_i$ to the start of $Q^+_i$.

Now for each $i$ we join $x_i$ to $x'_i$ to $Q^-_i$ to $P^-_i$ to $Q'_i$ to $P^+_i$ to $Q^+_i$ to $y'_i$ to $y_i$ in order to obtain the required vertex-disjoint paths from the $x_i$s to the $y_i$s.
\end{proof}

\subsection*{Acknowledgement}
The author would like to thank Codru\cb{t} Grosu for suggesting a simplification in the proof of Lemma~\ref{LargeLinkage}.

\bibliography{Linking}
\bibliographystyle{abbrv}
\end{document}